\documentclass[12pt]{amsart}
\usepackage{amsfonts}
\usepackage{amssymb}
\usepackage{color}
\usepackage{blindtext}
\usepackage{mathtools}
\usepackage{bbm}

\makeatletter \@addtoreset{equation}{section}

\makeatother \makeatletter \normalbaselineskip=14pt
\normalbaselines \headsep10mm
\def\OL{\relax\ifmmode {\sf L}\else{\textsf L}\fi}
\def\OR{\relax\ifmmode {\sf R}\else{\textsf R}\fi}

\newcommand{\mb}{\mathbb}
\newcommand{\mc}{\mathcal}
\newcommand{\eul}{\mathfrak}

\newcommand{\bou}{_{\scriptscriptstyle{\rm b}}}

\newcommand{\A}{\eul A}
\newcommand{\Ao}{{\eul A}_{\scriptscriptstyle 0}}

\newcommand{\vp}{\varphi}

\newcommand{\mult}{\,{\scriptstyle \square}\,}
\newcommand{\D}{{\mc D}}

\def\x{\relax\ifmmode {\mbox{*}}\else*\fi}

\newcommand{\id}{\mathbbm{1}}
\newcommand{\ip}[2]{\langle{#1}|{#2}\rangle}
\newcommand{\ipp}[2]{{#1}|{#2}\rangle}

\newtheorem{defn}{Definition}[section]
\newtheorem{prop}[defn]{Proposition}
\newtheorem{thm}[defn]{Theorem}
\newtheorem{lemma}[defn]{Lemma}

\theoremstyle{remark}
\newtheorem{rem}[defn]{Remark}
\newtheorem{example}[defn]{Example}
\newcommand{\bedefi}{\begin{defn}$\!\!${\bf }$\;$\rm }
\newcommand{\findefi}{ \end{defn}}

\newcommand{\SSA}{{\mathcal S}_{\Ao}(\A)}

\newcommand{\betheo}{\begin{thm}}
\newcommand{\entheo}{\end{thm}}

\newcommand{\becor}{\begin{coroll}}
\newcommand{\encor}{\end{coroll}}

\newcommand{\belem}{\begin{lemma}}
\newcommand{\enlem}{\end{lemma}}

\newcommand{\beprop}{\begin{prop}}
\newcommand{\enprop}{\end{prop}}
\newcommand{\berem}{\begin{rem}$\!\!${\bf }$\;$\rm }
\newcommand{\beex}{\begin{example}$\!\!${\bf }$\;$\rm }
\newcommand{\enex}{ \end{example}}
\newcommand{\enrem}{ \end{rem}}

\newcommand{\odelta}{\overline{\delta}}

\def\H{{\mathcal H}}
\newcommand{\wmult}{\mbox{\raisebox{1pt}{$\scriptscriptstyle{
\square}$}}}

\newcommand{\e}{{\sf e}}

\begin{document}
\title[Derivations and automorphism groups]{Unbounded derivations and *-automorphisms groups of Banach quasi *-algebras}
\author{Maria Stella Adamo}
\author{Camillo Trapani}

\address{\textsc{Maria Stella Adamo}, Dipartimento di Matematica e Informatica, Universit\`a di Palermo, I-90123 Palermo, Italy}  \email{mariastella.adamo@community.unipa.it; msadamo@unict.it}
\address{\textsc{Camillo Trapani}, Dipartimento di Matematica e Informatica, Universit\`a di Palermo, I-90123 Palermo, Italy} \email{camillo.trapani@unipa.it}
\date{\today}
\maketitle

\textsc{\abstractname{}}: This paper is devoted to the study of unbounded derivations on Banach quasi *-algebras with a particular emphasis to the case when they are infinitesimal generators of one parameter automorphisms groups.
Both of them, derivations and automorphisms are considered in a weak sense; i.e., with the use of a certain families of bounded sesquilinear forms. Conditions for a weak *-derivation to be the generator of a *-automorphisms group are given.

\section{Introduction}\label{Sec1}
The aim of this paper is to investigate unbounded derivations on Banach quasi *-algebras, focusing, in particular, on the case in which derivations arise as infinitesimal generators of one parameter *-au\-to\-mor\-phisms groups. In the case of  Banach *-algebras, a derivation is a linear map defined on a *-subalgebra for which the Leibnitz rule holds. These objects have a strict relation with several branches of Mathematics and Physics, especially with quantum theories. On one hand they naturally appear in the representation theory of Lie algebras and, on the other hand, commutators (the prototypes of derivations) and related commutation relations constitute the cornerstone of the mathematical description of quantum theories.

Derivations on C*-algebras, Banach *-algebras or even more general structures as partial *-algebras  have been extensevely studied in order to answer questions about closability, spatiality or time-evolution systems stemming out by derivations through one parameter automorphisms groups, often under different points of view (see \cite{Ant3,Ant4,Ant2,Bag1,Brat1,Brat2,Kish,Sakai,weigt,WZ1,WZ2}).

A Banach quasi *-algebra $(\A, \Ao)$ arises in natural way as the completion, $\A$,  of a given *-algebra $\Ao$ under a norm topology, in the case when the multiplication is not jointly continuous and, for this reason, it has the peculiarity that the product is defined only for selected couples of elements: it is in other words, a partial *-algebra  (see \cite{AT,Ant1,Bag3}). The possible lack of an everywhere defined multiplication reveal new features and requires a convenient adaptation of very familiar concepts. In the case of derivations a weaker concept is needed, which is introduced by making use of certain families of bounded positive invariant  sesquilinear forms.

In this paper, we examine some properties of weak *-derivations on Banach quasi *-algebras, in particular those that characterize them as generators of one parameter groups of weak *-automorphisms. We apply our results in order to study examples of derivations not treatable in the classical algebraic background.

In details, in Section 2 we remind some definitions, properties and preliminary results useful in what follows. Passing to Section 3, densely defined derivations on a *-semisimple Banach quasi *-algebra $(\A,\Ao)$ are investigated, starting from inner qu*-derivations, i.e. those that can be written, for a fixed $h\in\A$, as $\delta_h(x)=i[h,x]$ with $x\in\Ao$. In this case, it is shown that every inner qu*-derivation is closable, not depending on the nature of the element $h$. 
Most likely, the closure is not again a derivation in the usual sense for the failure of the Leibnitz rule. Therefore, we employ sesquilinear forms to find a good candidate for a weak Leibnitz rule. What we achieve is indeed more general and it is examined in details.

In Section 4 we introduce weak *-automorphisms of a Banach quasi *-algebra and derive some properties of the generator of a one parameter group of weak*-automorphisms. 

In Section 5 we investigate  the possibility of extending to the case of *-semisimple Banach quasi *-algebras the well known result of Bratteli and Robinson theorem for C*-algebras \cite{Brat2} about conditions on a *-derivation for being the infinitesimal generator of a one-parameter group of the type described above. 

Finally, in Section 6, we apply our results to compute the one parameter group generated by a inner qu*-derivation and give a physical example (see \cite{Bag2,Bag4}) that motivates our choice to examine derivations in a more general context when the implementing element is unbounded.

\section{Preliminaries and basic results}\label{Sec2}
In this section we give some definitions and preliminary results need\-ed along the paper. For more details, see \cite{Ant1}.

\subsection{Basic definitions}
A {\em quasi *--algebra} $(\A, \Ao)$ is a pair consisting of a vector space $\A$ and a *--algebra $\Ao$ contained in $\A$ as a subspace and such that
\begin{itemize}
\item[(i)] $\A$ carries an involution $a\mapsto a^*$ extending the involution of $\Ao$;
\item[(ii)] $\A$ is  a bimodule over $\A_0$ and the module multiplications extend the multiplication of $\Ao$. In particular, the following associative laws hold:
\begin{equation}\notag \label{eq_associativity}
(xa)y = x(ay); \ \ a(xy)= (ax)y, \; \forall \, a \in \A, \,  x,y \in \Ao;
\end{equation}
\item[(iii)] $(ax)^*=x^*a^*$, for every $a \in \A$ and $x \in \Ao$.
\end{itemize}

A quasi *-algebra $(\A, \Ao)$ is \emph{unital} if there is an element $\mathbbm{1}\in \Ao$, such that $a\mathbbm{1}=a=\mathbbm{1} a$, for all $a \in \A$; $\mathbbm{1}$ is unique and called the \emph{unit} of $(\A, \Ao)$.

\bedefi \label{qu_deriv}
Let $(\A,\Ao)$ be a quasi *-algebra and $\delta$ a linear map of $\Ao$ into $\A$. We say that $\delta$ is a {\em qu*-derivation} of $(\A,\Ao)$ if
\begin{itemize}
\item[(i)]$\delta(x^*)=\delta(x)^*, \; \forall x \in \Ao$
\item[(ii)]$\delta(xy)=\delta(x)y + x\delta(y),\; \forall x,y \in \Ao$
\end{itemize}
\findefi

\beex \label{ex_spatial} The easiest example of a qu*-derivation on a quasi *-algebra is provided by the {\em commutator}; i.e., if $h=h^*\in \A$ we put
$$ \delta(x)= i [h,x]:=i(hx-xh).$$
\enex

Motivated by this example we give the following
\bedefi Let $(\A,\Ao)$ be a quasi *-algebra and $\delta$ a qu*-derivation of $(\A,\Ao)$. We say that $\delta$ is {\em inner} if there exists
$h=h^*\in \A$ such that
$$ \delta(x)= i [h,x].$$
\findefi

The framework of our whole discussion will be Banach quasi *-algebras (with particular attention to the *-semisimple case).
We remind the definition.

\bedefi A  quasi *-algebra $(\A,\Ao)$ is called a {\em normed quasi *-algebra} if a norm
$\|\cdot\|$ is defined on $\A$ with the properties
\begin{itemize}

\item[(i)]$\|a^*\|=\|a\|, \quad \forall a \in \A$;
\item[(ii)] $\Ao$ is dense in $\A$;
\item[(iii)]for every $x \in \Ao$, the map $R_x: a \in \A \to ax \in \A$ is continuous in
$\A$.
\end{itemize}
If $(\A,\| \cdot \|) $ is a Banach space, we say that $(\A,\Ao)$ is a {\em Banach quasi *-algebra}.\label{def} The norm topology of $\A$ will be denoted by $\tau_n$.
\findefi The continuity of the involution implies that
\begin{itemize}
\item[(iii')]for every $x \in \Ao$, the map $L_x: a \in \A \to xa \in \A$ is continuous in
$\A$.
\end{itemize}
If $x \in \Ao$, we put
$$\|x\|_0:= \max\left\{\sup_{\|a\|\leq 1}\|ax\|, \sup_{\|a\|\leq 1}\|xa\|\right\}.$$
Then,
\begin{align*}
&\|ax\|\leq \|a\|\|x\|_0, \quad \forall a \in \A, \, x \in \Ao;\\
&\|x^*\|_0 = \|x\|_0, \quad \forall  x \in \Ao.
\end{align*}

Let $(\A, \Ao)$ be a Banach quasi *-algebra. If $\Ao[\|\cdot\|_0]$ is a C*-algebra, then $(\A,\Ao)$ is called {\em proper CQ*-algebra}.
\medskip

If $E,F$ are Banach spaces, $\mathcal{D}(E)$ is a dense subspace of $E$ and $S:\mathcal{D}(E)\subset E\to F$ is a linear operator, then we define the following important subsets of the complex field
\begin{itemize}
\item the \textit{resolvent} $\rho(S):=\{\lambda \in\mathbb{C}:\exists(\lambda I-S)^{-1}\;\text{and it is bounded}\}$;
\item the \textit{spectrum} $\sigma(S):=\mathbb{C}\setminus\rho(S)$.
\end{itemize}

\subsection{*-Semisimple Banach quasi *-algebras}

\bedefi\label{semisim}
Let $(\A,\Ao)$ be a Banach quasi *-algebra and denote as $\mathcal{S}_{\Ao}(\A)$
the set of all sesquilinear forms $\varphi$ on $\A\times\A$ that satisfy the following conditions
\begin{itemize}
\item[(i)] $\varphi(a,a)\geq0$ for every $a\in\A$
\item[(ii)] $\varphi(ax,y)=\varphi(x,a^*y)$ for every $a\in\A$ and $x,y\in\Ao$
\item[(iii)]$|\varphi(a,b)|\leq \|a\|\|b\|$, for every $a,b, \in \A$.
\end{itemize}
The Banach quasi *-algebra $(\A,\Ao)$ is called \textit{*-semisimple} if
$$\mathcal{R}^*=\{a:\varphi(a,a)=0,\;\forall\varphi\in\SSA\}=\{0\}.$$
\findefi

According to the following Lemma, $\mathcal{R}^*$ can be characterized into different ways

\begin{lemma}\cite[Lemma 4.1]{Bag3}\label{lemma} Let $(\A,\Ao)$ be a Banach quasi *-algebra and let us consider the following sets
\begin{itemize}
\item $\mathcal{R}^*=\{a\in\A:\varphi(a,a)=0,\;\forall\varphi\in\SSA\}$
\item $\mathcal{R}_1=\{a\in\A:\varphi(ax,y)=0,\;\forall x,y\in\Ao,\forall\varphi\in\SSA\}$
\item $\mathcal{R}_2=\{a\in\A:\varphi(ax,ay)=0,\;\forall x,y\in\Ao,\forall\varphi\in\SSA\}.$
\end{itemize}
Then $\mathcal{R}^*=\mathcal{R}_1=\mathcal{R}_2$.
\end{lemma}

By means of the sesquilinear forms $\varphi\in\SSA$ we can define a new multiplication in $\A$ between couples of elements none of them belonging necessarily to $\Ao$.

\bedefi \cite{Frag2}\label{weak_mult} Let $(\A,\Ao)$ be a *-semisimple Banach quasi *-algebra. Let $a,b\in\A$. We say that the {\it weak} multiplication $a\wmult b$ is well-defined if there exists a (necessarily unique) $c\in\A$ such that:
$$
\varphi(bx,a^*y)=\varphi(cx,y),\;\forall\, x,y\in\Ao, \forall\,\varphi\in\SSA.
$$
In this case, we put $a\wmult b:=c$.
\findefi

The following result is immediate.

\begin{prop}
Let $(\A,\Ao)$ be a *-semisimple Banach quasi *-algebra. Then $\A$ endowed with the weak multiplication $\wmult$ is a partial *-alge\-bra.
\end{prop}
We will denote by $R_w(\A)$ (resp. $L_w(\A)$) the space of universal right (resp. left) weak multipliers of $\A$; i.e., the space of all $b\in \A$ such that $a\wmult b$ (resp. $b\wmult a$) is well-defined, for every $a\in \A$. Clearly, $\Ao \subseteq L_w(\A) \cap R_w(\A)$.

Several results are known about the weak multiplication we defined above (see \cite{Frag2}). {As we shall see below the weak multiplication can be characterized through some closedness properties with respect to appropriate topologies defined by means of the sesquilinear forms $\varphi\in\SSA$.}

\bedefi
Let $(\A,\Ao)$ be a *-semisimple Banach quasi *-algebra. The sesquilinear forms $\varphi$ of $\SSA$ define the topologies generated by the following families of seminorms:
\begin{itemize}
\item[$\tau_w $:]
$\quad a\mapsto |\varphi(a x,y)|$, $\quad\varphi\in\SSA, x,y\in\Ao$;
\item[$\tau_s $:]
$\quad a\mapsto \varphi(a,a)^{1/2}$, $\quad\varphi\in\SSA$;
\item[$\tau_{s^*} $:]
$\quad a\mapsto \max\{\varphi(a,a)^{1/2},\varphi(a^*,a^*)^{1/2}\}, \quad\varphi\in\SSA$.
\end{itemize}
\findefi

\berem
From the continuity of $\varphi\in\SSA$ it follows that all
the topologies $\tau_w$, $\tau_s$, (and also $\tau_{s^*}$, if the
involution is $\tau$-continuous) are coarser than the initial norm topology of $\A$. \enrem

It is easy to show that

\begin{prop}\label{prop_43} The following statements are equivalent.
\begin{itemize}
\item[(i)] The weak product $a\wmult b$ is well defined.
\item[(ii)] There exists a sequence $\{y_n\}$ in $\Ao$ such that $\|y_n- b\|\to 0$
and $a y_n\stackrel{\tau_w}\longrightarrow c\in\A$.
\item[(iii)] There exists a sequence $\{x_n\}$ in $\Ao$ such that $\|x_n- a\|\to 0$
and $x_n b\stackrel{\tau_w}\longrightarrow c\in\A$.
\end{itemize}
\end{prop}

\berem In Proposition \ref{prop_43}, if $a,b\in\A$ are such that $a\wmult b$ is well-defined, then \textit{every} sequence $\{y_n\}$ in $\Ao$ that approximates $b$ verifies condition (ii). Likewise, the same holds for a sequence $\{x_n\}$ is $\Ao$ such that $\|x_n-a\|\to0$.
\enrem

Let $(\A,\Ao)$ be a Banach quasi-*algebra. To every $a\in \A$ there correspond the linear maps $L_a$ and $R_a$ defined as
\begin{equation}\label{Left}L_a:\Ao\to\A\quad L_ax=ax\quad\forall x\in\Ao
\end{equation}
\begin{equation}\label{Right}
R_a:\Ao\to\A\quad R_ax=xa\quad\forall x\in\Ao
\end{equation}

If $(\A,\Ao)$ is a *-semisimple Banach quasi *-algebra, then the weak multiplication $\wmult$ allows us to extend $L_a$, (resp., $R_a$) to $R_w(a)$ (resp., $L_w(a)$).
Let us denote by $\hat{L}_a$, (resp. $\hat{R}_a$) these extensions. Then $\hat{L}_a b= a\wmult b$, for every $b \in R_w(a)$ and  $\hat{R}_a c= c\wmult a$, for every $c \in L_w(a)$.
\begin{prop}\label{prop_415}
Let $(\A,\Ao)$ be a *-semisimple Banach quasi *-algebra. Then,
for every $a \in \A$, $\hat{L}_a, \hat{R}_a$ are closed linear maps in $\A[\tau_n]$.
\end{prop}

\berem Proposition \ref{prop_415} implies, obviously, that $L_a$, $R_a$ are closable and $\overline{L}_a \subseteq \hat{L}_a$, $\overline{R}_a \subseteq \hat{R}_a$
\enrem
\bedefi\label{bounded} Let $(\A,\Ao)$ be a Banach quasi-*algebra. An element $a\in\A$  is called \textit{bounded} if the operators $L_a$ and $R_a$ defined in (\ref{Left}) and (\ref{Right}) are $\|\cdot\|-$continuous, thus extendible to the whole space $\A$.
The set of bounded elements is denoted by $\A_{\bou}$.
\findefi
\begin{lemma}\label{lemma_univmult} If $(\A,\Ao)$ is a *-semisimple Banach quasi-*algebra with unit $\id$, the set $\A_{\bou}$ of bounded elements is a *semisimple Banach algebra.
Moreover, $\A_{\bou}$ coincides with the set $R_w(\A)\cap L_w(\A)$.
\end{lemma}
\begin{proof} The first statement and the inclusion $\A_{\bou}\subset R_w(\A)\cap L_w(\A)$ were shown in \cite{Ant1}.  Let $a\in R_w(\A)\cap L_w(\A)$ then $R_w(a)=L_w(a)=\A$. Thus $\hat{L}_a$ (resp., $\hat{R}_a$)  is closed and everywhere defined. Hence both $L_a$ and $R_a$ are bounded. \end{proof}

{If $(\A,\Ao)$ is a *-semisimple Banach quasi *-algebra with unit $\id$ and $a\in\A$, we define
\begin{itemize}
\item the resolvent $\rho(a)=\{\lambda \in {\mb C}:\exists (\lambda\id-a)^{-1} \in \A_{\bou} \}= \rho(\overline{L}_a)\cap \rho(\overline{R}_a) $
\item the spectrum $\sigma(a)={\mb C}\setminus\rho(a)$
\end{itemize}
As shown in \cite{ct1}, $\sigma(a)$ is a bounded subset of ${\mb C}$ if, and only if, $a\in \A_{\bou} $.}

\section{Extensions of qu*-derivations}\label{Sec3}
Let $(\A,\Ao)$ be a Banach quasi *-algebra. By the very definitions, $\A$ is a \textit{Banach module over $\Ao$}. Therefore, if $\Ao[\|\cdot\|_0]$ is a C*-algebra, or, in other words if
$(\A,\Ao)$ is a \textit{proper CQ*-algebra}, introduced in Section \ref{Sec2},
 then by Ringrose's theorem \cite{Ring}, $\delta$ is continuous from $\Ao[\|\cdot\|_0]$ into $\A[\|\cdot\|]$.
But we may also regard $\Ao$ as a $\|\cdot\|-$dense subspace of $\A$. Hence it is natural to pose the question as to whether $\delta$ can be extended beyond $\Ao$.

If $\delta$ is closable as a linear map from $\Ao[\|\cdot\|]$ into $\A[\|\cdot\|]$, then as usual we can define its closure $\odelta$ on the subspace
\begin{multline}\notag D(\odelta)=\{a\in \A: \exists \{x_n\}\subset \Ao, w\in \A \mbox{ s.t. } \|a-x_n\|\to 0 \mbox{ and } \\
\|\delta(x_n)-w\|\to 0\}.
\end{multline}
by
$$\odelta(a):=w= \lim_{n\to\infty}\delta(x_n).$$

In general, $\odelta$ is not a *-derivation, since $(D(\odelta),\Ao)$ is not necessarily a quasi *-algebra and the Leibniz rule may fail (see Example \ref{ex_noex}).
\medskip

The simplest case we could investigate is that of an inner qu*-derivation: not surprisingly, $\delta_h$ is continuous whenever the element $h\in\A$ that generates the qu*-derivation is bounded in the sense of Definition \ref{bounded}.

\begin{prop}\label{clos_inn}
Let $(\A,\Ao)$ be a *-semisimple Banach quasi *-algebra. Let $h\in\A$ be a fixed  element in $\A$ such that $h=h^*$ and $\delta_h$ the qu*-derivation defined as $\delta_h(x):=i[h,x]$ for $x\in\Ao$. Then $\delta_h$ is closable.
\end{prop}
\begin{proof}
Let $\{x_n\}\subset\Ao$ be a sequence that vanishes as $n\to\infty$ and such that $\delta_h(x_n)$ is $\|\cdot\|-$Cauchy, i.e. there exists $w\in\A$ such that $\|\delta_h(x_n)-w\|\to0$ as $n\to+\infty$. We want to show that $w=0$.

On one hand, for every $\varphi\in\SSA$ and for every $u,v\in\Ao$,
\begin{align*}
\varphi(\delta_h(x_n)u,v)&=i\varphi(hx_nu,v)-i\varphi(x_nhu,v)\\
&=i\varphi(x_nu,hv)-i\varphi(hu,x_n^*v)\to0.
\end{align*}
On the other hand, by the hypotheses $\varphi(\delta_h(x_n)u,v)\to\varphi(wu,v)$, for every $\varphi\in\SSA$ and for every $u,v\in\Ao$. We conclude by Lemma \ref{lemma} and the arbitrary choice of $\varphi\in\SSA$.
\end{proof}

Let us now assume that $\delta$ is a closable qu*-derivation. We consider the question as to whether its closure $\odelta$ is a *-derivation in some weaker sense; i.e.; if a sort of Leibniz rule still hold.

{\begin{prop}\label{prop_123} Let $(\A,\Ao)$ be a *-semisimple Banach quasi *-algebra with $\A[\tau_w]$ sequentially complete. Let $\delta$ be a closable qu*-derivation of $(\A,\Ao)$ with closure $\odelta$. Then, if $a,b\in\mathcal{D}(\odelta)$ and $a\wmult b$ is well-defined, there exists an element $\odelta_w(a\wmult b)\in\A$ such that
$$\varphi(\odelta_w(a\wmult b)u,v)=\varphi(bu,\odelta(a)^*v)+\varphi(\odelta(b)u,a^*v)\quad\forall u,v\in\Ao,\varphi\in\SSA.$$
\end{prop}
\begin{proof} 
Suppose that $\delta$ is a closable qu*-derivation and let $\odelta$ be its closure. Let $a,b\in\mathcal{D}(\odelta)$, then there exist sequences $\{x_n\},\{y_n\}$ of elements in $\Ao$ such that $\|x_n-a\|\to0$, $\|y_n-b\|\to0$ and $\|\delta(x_n)-\odelta(a)\|\to0$, $\|\delta(y_n)-\odelta(b)\|\to0$.

The sequence $\{x_ny_n\}\subset \Ao$ is $\tau_w-$convergent to $a\wmult b$ and $\{\delta(x_ny_n)\}$ is $\tau_w-$Cauchy. By the sequential completeness of $\A[\tau_w]$, there exists $c\in\A$ such that $\delta(x_ny_n)\xrightarrow[]{\tau_w}c$. Computing the $\tau_w-$limit
\begin{align*}
\varphi(\delta(x_ny_n)u,v)&=\varphi(\delta(x_n)y_nu,v)+\varphi(x_n\delta(y_n)u,v)\\
&=\varphi(y_nu,\delta(x_n)^*v)+\varphi(\delta(y_n)u,x_n^*v)\\
&\to\varphi(bu,\odelta(a)^*v)+\varphi(\odelta(b)u,a^*v),
\end{align*}
for every $u,v\in\Ao$, for all $\varphi\in\SSA$, we obtain
$$\odelta_w(a\wmult b):=c=\tau_w-\displaystyle\lim_{n\to\infty}\delta(x_ny_n)$$
and therefore
$$\varphi(\odelta_w(a\wmult b)u,v)=\varphi(bu,\odelta(a)^*v)+\varphi(\odelta(b)u,a^*v)\quad\forall\varphi\in\SSA,u,v\in\Ao.$$
\end{proof}

\beex
Consider the Banach quasi *-algebra $(L^p(\mathbb{R}),\mathcal{C}^{\infty}_c(\mathbb{R}))$. For $p\geq2$ $(L^p(\mathbb{R}),\mathcal{C}^{\infty}_c(\mathbb{R}))$ is a *-semisimple Banach quasi *-algebra. Define on $\mathcal{C}^{\infty}_c(\mathbb{R})$ the derivation $\delta(f)=f'$ for every $f\in\mathcal{C}^{\infty}_c(\mathbb{R})$, where $f'$ is the classical derivative of $f$.
{Then $\delta$ is closable and its closure} is the \textit{weak derivative} in $W^{1,p}(\mathbb{R})$. This is a favourable situation because $W^{1,p}(\mathbb{R})$ is a *-algebra and the Leibniz rule works for the weak  derivative.
\enex
\beex Working again with the derivative operator, we can construct an example where the domain of the closure of a qu*-derivation is a genuine quasi *-algebra.
Let
\begin{align*} D(S):= \{f\in L^2([0,1]): &f \mbox{ absolutely continuous,}\\ &f'\in L^2([0,1]), \, f(0)=f(1) \}\end{align*}
and define $Sf=f'$. $D(S)$ is a *-algebra with the usual operations and conjugation of functions and $S^*=-S$.
$D(S)$ is also a Hilbert space with the graph norm $\|f\|_S= (\|f\|^2 + \|Sf\|^2)^{1/2}.$ Let us denote it by $\H_S$. If $\H_S^\times$ denotes its conjugate dual, which is a Hilbert space endowed with the usual dual norm,
then $S$ maps $\H_S$ into $\H_S^\times$ continuously. In particular, $S:\H_S\to L^2([0,1])$ continuously, then it has a continuous extension $\hat{S}: L^2([0,1])\to \H_S^\times$ defined by $\ip{\hat{S}f}{g}=-\ip{f}{Sg}$ for $f\in L^2([0,1])$ and $g\in \H_S$.

We may also regard $\hat{S}$ as a densely defined operator in $\H_S^\times$.Then $S$ is closable since its hilbertian adjoint in $\H_S^\times$, denoted by $\hat{S}^\times$, is densely defined; in fact $\hat{S}^\times \supset -S$. It is clear that $S$ is a qu*-derivation of the Banach quasi *-algebra $(\H_S^\times, \H_S)$. Moreover, the inclusions $S\subset \hat{S}$ and $S\subset \hat{S}^\times$ imply that $D(S^{\times\times})= D(\hat{S}^{\times\times}) \supset L^2([0,1])$. But $\hat{S}=-S^\times$, by definition. Hence $D(S^{\times\times})=D(\overline{S})=D(\hat{S})=L^2([0,1])$, $\overline{S}$ denotes here the closure of $S$ in $\H_S^\times$. It is easily seen that $(L^2([0,1]), \H_S)$ is a quasi *-algebra and that the Leibniz rule works in this case. It is interesting to compute explicitly the action of $\hat{S}$ on some elements of $D(\overline{S})$.
For instance, if $f\in C^1([0,1])$ we can compute explicitly $\hat{S}f$. We have in fact, for every $g\in D(S)$
\begin{align*} \ip{\hat{S}f}{g}=-\ip{f}{Sg}&=-\int_0^1 f(x)\overline{g'(x)}dx\\& = -(f(1)-f(0))\overline{g(0)} +\int_0^1 f'(x)\overline{g(x)}dx\\& = \ip{f'-(f(1)-f(0))\delta_0}{g},\end{align*}
where $\delta_0$ denotes here the Dirac delta functional centered at $0$.
\enex
{
\beex \label{ex_noex} Let $(\H,\Ao)$ be a Hilbert quasi *-algebra; that is, $\H$ is the Hilbert space completion of a Hilbert algebra $\Ao$ with inner product $\ip{\cdot}{\cdot}$ (see, for instance, \cite{AT}). Let $\delta:\Ao\to \H$ be a qu*-derivation.
Then $\delta$ can be regarded as a densely defined linear operator in $\H$. To avoid possible confusion, if $T:\Ao\to \H$, we denote by $T^\star$ its adjoint. The qu*-derivation $\delta$ is closable if, and only if, it admits a densely defined adjoint $\delta^\star$.  Let $D(\delta^\star)$ denote the domain of $\delta^\star$. It is easy to check that $\xi \in D(\delta^\star)$ if, and only if, $\xi^*\in D(\delta^\star)$.
Let us assume that $\Ao \subset D(\delta^\star)$. Then, for every $x,y,z \in \Ao$, the following equality holds
$$ \ip{\delta(x)}{yz}= \ip{x}{\delta^\star(y)z}-\ip{x}{y\delta(z)}.$$
Hence $\delta^\star(yz)= \delta^\star(y)z -y\delta(z)$ for every $y,z \in \Ao$. Thus $\delta^\star$ is not, in general, a qu*-derivation. A necessary and sufficient condition for this to be true is that $\delta^\star(x)=-\delta(x)$, for every $x\in \Ao$. This is exactly what happens when $\delta=\delta_h$, where $h=h^*\in \A$ and $\delta_h(x)=i[h,x]=i(hx-xh)$, $x\in \Ao$. An easy computation shows, in fact, that \mbox{$\delta_h\subset -\delta_h^\star$}. 

Let $\xi\in D(\delta^\star)$ and $y\in \Ao$. then, for every $x\in \Ao$ we get
$$\ip{\delta(x)}{\xi y}= \ip{x}{\delta^\star(\xi)y}-\ip{x\delta(y^*)}{\xi}=\ip{x}{\delta^\star(\xi)y}-\ip{R_{\delta(y^*)}x}{\xi}.$$
From these equalities we deduce that $(D(\delta^\star), \Ao)$ is a quasi *-algebra if, and only if,
\begin{equation}\label{eqn_dom_star} D(\delta^\star)\subset \bigcap_{y\in \Ao}D(R_{\delta(y^*)}^\star).\end{equation}

The closure $\overline{\delta}$  clearly coincides with $\delta^{\star\star}$.
Let $\xi\in D(\overline{\delta})$; then there exists a sequence$\{z_n\}\subset \Ao$, such that $z_n\to \xi$ and $\delta{(z_n)} \to \overline{\delta}(\xi).$ Then, if $\eta \in D(\delta^\star)$ and $y\in \Ao$,
\begin{align*}\ip{\delta^\star(\eta)}{\xi y}&= \lim_{n\to \infty} \ip{\delta^\star(\eta )}{z_ny}=\lim_{n\to \infty} \ip{\eta}{\delta(z_n y)}\\
&=\lim_{n\to \infty} \ip{\eta}{\delta(z_n)y+ z_n\delta(y)} = \ip{\eta}{\overline{\delta}(\xi)y}+\lim_{n\to \infty} \ip{\eta}{z_n \delta(y)}.
\end{align*}
From these equalities it follows that $\xi y\in D(\delta^{\star\star})=D(\overline{\delta})$ if, and only if, $\xi\in D(\overline{R_{\delta(y)}})$. Therefore
$(D(\overline{\delta}), \Ao)$ is a quasi *-algebra if, and only if,
\begin{equation}\label{eqn_dom} D(\overline{\delta})\subset \bigcap_{y\in \Ao}D(\overline{R_{\delta(y)}}).\end{equation}
\enex
}

\section{Weak derivations on *-semisimple Banach quasi *-algebras}\label{Sec4}
Derivations, in many occurrences, stem out as generators of automorphism groups. In this section we will start our analysis taking first of all,  this point of view.
The discussion in the previous section makes clear that the existence of  \textit{sufficiently many sesquilinear forms $\varphi$ to work with} is crucial when dealing with this problem. For this reason, from now on we will focus our attention to the case of *-semisimple Banach quasi *-algebras.

\subsection{Infinitesimal generators of weak *-automorphism groups} As it is known, in the case of C*-algebras *-derivations arise as infinitesimal generators of *-automorphisms groups.
Let us examine this aspect.
\bedefi Let $(\A,\Ao)$ be a *-semisimple Banach quasi *-algebra and $\theta:\A\to \A$ a linear bijection. We say that $\theta$ is a weak *-automorphism of $(\A,\Ao)$ if
\begin{itemize}
\item[(i)] $\theta(a^*)=\theta(a)^*$, for every $a \in \A$;
\item[(ii)] $\theta(a)\wmult \theta(b)$ is well defined if, and only if, $a\wmult b$ is well defined and, in this case, $$\theta(a\wmult b)=\theta(a)\wmult \theta(b).$$\end{itemize}
\findefi

By the previous definition, if $\theta$ is a weak *-automorphism, then $\theta^{-1}$ is a weak *-automorphism too.

{
{\begin{lemma}\label{im_theta} If $\theta$ is a weak*-automorphism of a a *-semisimple Banach quasi *-algebra $(\A,\Ao)$, then $\theta(\A_{\bou})=\A_{\bou}$.
\end{lemma}
\begin{proof} Let $a\in \A_{\bou}$, then $a\wmult b$ is well defined for every $b\in \A$. Hence, $\theta(a)\in R_w(\theta(\A))=R_w(\A)$. Similarly, $\theta(a)\in L_w(\A)$. Thus $\theta(a) \in \A_{\bou}$, by Lemma \ref{lemma_univmult}. Applying this result to $\theta^{-1}$ one gets the converse inclusion.
\end{proof}}

}
\bedefi Let $(\A,\Ao)$ be a *-semisimple Banach quasi *-algebra. Suppose that for every fixed $t\in {\mb R}$, $\beta_t$ is a weak *-automorphism of $\A$. If
 \begin{itemize} \item[(i)]$\beta_0(a)=a,$ $\forall a\in \A$  \item[(ii)] $\beta_{t+s}(a)= \beta_t(\beta_s(a))$, $\forall a\in \A$ \end{itemize}then we say that $\beta_t$ is a {\em one-parameter group of weak *-automorphisms of $(\A,\Ao)$}.
 If $\tau$ is a topology on $\A$ and the map $t\mapsto \beta_t(a)$ is $\tau$-continuous, for every $a\in \A$, we say that $\beta_t$ is a $\tau$-continuous weak *-automorphism group.
\findefi

The definition of the infinitesimal generator of $\beta_t$ is now quite natural. If $\beta_t$ is
$\tau$-continuous, we set

$$ \mathcal{D}(\delta_\tau)=\left\{a\in \A: \lim_{t\to 0} \frac{\beta_t(a)-a}{t} \mbox{ exists in $\A[\tau]$}\right\}$$
and
$$ \delta_\tau (a)=\tau-\lim_{t\to 0} \frac{\beta_t(a)-a}{t}, \quad a \in  \mathcal{D}(\delta_\tau).$$

If the involution $a\mapsto a^*$ is $\tau$-continuous, then $a\in \mathcal{D}(\delta_\tau)$ implies $a^*\in \mathcal{D}(\delta_\tau)$ and
$\delta(a^*)=\delta(a)^*$. Clearly, $\mathcal{D}(\delta_{\tau_n}) \subseteq \mathcal{D}(\delta_{\tau_{s^*}}) \subseteq \mathcal{D}(\delta_{\tau_w})$.

Of course, one would expect, in analogy with the C*-situation, that $\mathcal{D}(\delta_\tau)$ is a partial *-algebra and that $\delta$ is a *-derivation in a sense to be specified; in other words, we should
decide which form of Leibniz rule must be taken to define conveniently derivations on a partial
*-algebra. The following proposition suggests an answer to that question.
\begin{prop}\label{prop_38} Let $(\A,\Ao)$ be a *-semisimple Banach quasi *-algebra and $\beta_t$ a $\tau_{s^*}$-continuous weak *-automorphism group of $(\A,\Ao)$.
Then the following statements hold.
\begin{itemize}
\item[(i)] $\delta_{\tau_{s^*}}(a^*)=\delta_{\tau_{s^*}}(a)^*$ for all $a \in \mathcal{D}( \delta_{\tau_{s^*}})$;
\item[(ii)] If $a, b \in \mathcal{D}( \delta_{\tau_{s^*}})$ and $a\wmult b$ is well defined, then $a\wmult b\in \mathcal{D}( \delta_{\tau_{w}})$ and 
\begin{align*}\vp(\delta_{\tau_{w}}(a\wmult b)x,y)&=\vp(bx,\delta_{\tau_{s^*}}(a)^*y)+\vp(\delta_{\tau_{s^*}}(b)x,a^*y),\\ & \forall a,b \in \mathcal{D}( \delta_{\tau_{s^*}}),\, a\in L_w(b); x,y \in \Ao.\end{align*}
\item[(iii)] If $\mathcal{D}( \delta_{\tau_{w}})=\mathcal{D}( \delta_{\tau_n})$ then $\mathcal{D}( \delta_{\tau_n})$ is a partial *-algebra with respect to the weak multiplication.
\end{itemize}
\end{prop}

\begin{proof} We prove only (ii). Let $a, b \in \mathcal{D}( \delta_{\tau_{s^*}})$, with $a\in L_w(b)$. If $x,y\in \Ao$, then
\footnotesize{\begin{align*}
\lim_{t\to 0} \vp\left(\frac{\beta_t(a\wmult b)-a\wmult b}{t}x,y\right) &= \lim_{t\to 0}\vp\left(\frac{\beta_t(a)\wmult\beta_t(b)-a\wmult b}{t}x,y\right)\\
&= \lim_{t\to 0}\frac{1}{t}\left[\vp((\beta_t(a)\wmult \beta_t(b))x,y) -\vp(\beta_t(b)x,a^*y) \right]\\
&+ \lim_{t\to 0}\frac{1}{t}\left[\vp(\beta_t(b)x,a^*y) -\vp(bx,a^*y) \right]
\end{align*}}
Now, for the first term on the right hand side, we have
\footnotesize{\begin{align*}
&\left| \frac{1}{t}[\vp((\beta_t(a)\wmult \beta_t(b))x,y) -\vp(\beta_t(b)x,a^*y)] -\vp(bx, \delta_{\tau_{s^*}}(a)^*y)\right|\\
&\leq \left|\vp\left(\beta_t(b)x,\frac{\beta_t(a)^*-a^*}{t}y \right)- \vp(\beta_t(b)x, \delta_{\tau_{s^*}}(a)^*y) \right| \\
&+ |\vp(\beta_t(b)x, \delta_{\tau_{s^*}}(a)^*y)-\vp(bx,\delta_{\tau_{s^*}}(a)^*y)|\\
&\leq \vp({\beta_t(b)x},{\beta_t(b)x})^{1/2}\vp\left(\frac{\beta_t(a)^*-a^*}{t}y- \delta_{\tau_{s^*}}(a)^*y,\frac{\beta_t(a)^*-a^*}{t}y- \delta_{\tau_{s^*}}(a)^*y\right)^{1/2}\\
&+ \vp((\beta_t(b)-b)x,\beta_t(b)-b)x)^{1/2}\vp(\delta_{\tau_{s^*}}(a)^*y,\delta_{\tau_{s^*}}(a)^*y)^{1/2}\to 0.
\end{align*}
}
because of the $\tau_{s^*}$-continuity of $\beta_t$ and of the involution.
As for the second term we have, taking into account that $b\in \mathcal{D}( \delta_{\tau_{s^*}})$,
$$\lim_{t\to 0}\frac{1}{t}\left(\vp(\beta_t(b)x,a^*y) -\vp(bx,a^*y)\right)= \vp(\delta_{\tau_{s^*}}(b)x,a^*y).$$

This proves at once that if $a,b\in \mathcal{D}( \delta_{\tau_{s^*}})$ and $a\wmult b$ is well-defined, then $a\wmult b \in \mathcal{D}( \delta_{\tau_{w}})$ and
$$\vp(\delta_{\tau_{w}}(a\wmult b)x,y)=\vp(bx,\delta_{\tau_{s^*}}(a)^*y)+\vp(\delta_{\tau_{s^*}}(b)x,a^*y), \quad \forall x,y \in \Ao.$$
\end{proof}

Proposition \ref{prop_38} suggests the following definition inspired by the one given in \cite{Ant3,Ant4} for partial *-algebras of unbounded operators.
\bedefi \label{defn_deriv}
Let $(\A,\Ao)$ be a *-semisimple Banach quasi *-algebra and $\delta$ a linear map of $\mathcal{D}(\delta)$ into $\A$, where $\mathcal{D}(\delta)$ is a partial *-algebra with respect to the weak multiplication $\wmult$. We say that $\delta$ is a {\em  weak *-derivation} of $(\A,\Ao)$ if
\begin{itemize}
\item[(i)] $\Ao\subset\mathcal{D}(\delta)$
\item[(ii)]$\delta(x^*)=\delta(x)^*, \; \forall x \in \Ao$
\item[(iii)] if $a,b\in\mathcal{D}(\delta)$ and $a\wmult b$ is well defined, then $a\wmult b\in\mathcal{D}(\delta)$ and
$$\vp(\delta(a\wmult b)x,y)= \vp(bx,\delta(a)^*y)+\vp(\delta(b)x,a^*y),$$
for all $\vp\in \SSA$, for every $x,y \in \Ao$.
\end{itemize}
\findefi

Clearly, every qu*-derivation is a weak *-derivation with the assumption $\mathcal{D}(\delta)=\Ao$.
\medskip

\beex The space $L^p({\mb R})$, $p\geq 1$, can be coupled with several *-algebras of functions (for instance, $C_c^\infty ({\mb R})$, $C_o ({\mb R})\cap L^p({\mb R})$, $W^{1,2} ({\mb R})$) to obtain a Banach quasi *-algebra.
 For $p\geq 2$, $(L^p({\mb R}), C_c^\infty ({\mb R}))$ is a *-semisimple Banach quasi *-algebra: the corresponding set $\SSA$ is given by the form $\vp_w$ defined for $w \in L^{\frac{p}{p-2}}({\mb R})$ (for $p=2$, $\frac{p}{p-2}=\infty$), $w \geq 0$,
$$ \vp_w (f,g)= \int_{\mb R} f(x) \overline{g(x)}w(x)dx.$$ The weak multiplication $f \wmult g$ is well defined if, and only if, $fg\in L^{p}({\mb R})$.
Let us define for $v\in {\mb R}$, $\beta_t(f)=f_t$ where $f_t(x)=f(x+t)$, $f\in L^p({\mb R})$. Then $\beta_t$ is a weak *-automorphisms group.
Its infinitesimal generator  is, formally, the derivative operator with domain $W^{1,p} ({\mb R})$.
If we change the *-algebra taking for instance $C_o ({\mb R})\cap L^p({\mb R})$ we see that the domain of $\delta$ does not contain $\Ao$, in general.
\enex

%
\section{Integrating weak *-derivations}
As it is known from Bratteli-Robinson theorem \cite{Brat1,Brat2, BratRob}) the fact that a *-derivation $\delta$ is closed is a necessary (albeit insufficient) condition for $\delta$ to be the generator of a norm continuous one-parameter group $\beta$ of *-automorphisms of a C*-algebra or, in other words, for $\delta$ to be {\em integrable}.  In this section we will prove analogous results in the case of *-semisimple Banach quasi *-algebras. The investigation performed in the previous sections on the closability of qu*- or weak derivation was, in a sense, preliminary for this scope.
\smallskip

\begin{thm}\label{HY1} Let $\delta:\mathcal{D}(\delta)\to\A[\|\cdot\|]$ be a weak *-derivation on a *-semisimple Banach quasi *-algebra $(\A,\Ao)$. Suppose that $\delta$ is the infinitesimal generator of a uniformly bounded, {$\tau_n$-continuous} group of weak *-automorphisms of $(\A,\Ao)$. Then $\delta$ is closed; its resolvent set $\rho(\delta)$ contains {$\mathbb{R}\setminus\{0\}$} and
\begin{equation}\label{eqn_lowbound}\|\delta(a)-\lambda a\|\geq|\lambda|\,\|a\|,\quad a\in\mathcal{D}(\delta),\lambda \in {\mb R}.\end{equation}
\end{thm}
Before proving Theorem \ref{HY1}, we first remind some properties on one-parameter groups $\{\beta_t\}_{t\in\mathbb{R}}$ that can be proved as in \cite{Pazy}.

\begin{lemma}\label{beta_t} Let $(\A,\Ao)$ be a *-semisimple Banach quasi *-algebra and let $\{\beta_t\}_{t\in\mathbb{R}}$ as in 1. of Theorem \ref{HY1}. Let $\delta$ be the infinitesimal generator of $\{\beta_t\}_{t\in\mathbb{R}}$. Then
\begin{enumerate}
\item for $a\in\A$
$$\|\cdot\|-\lim_{h\to0}\frac1h\int_t^{t+h}\beta_s(a)ds=\beta_t(a);$$
\item for $a\in\A$, $\int_0^t\beta_s(a)ds\in\mathcal{D}(\delta)$ and
$$\delta\left(\int_0^t\beta_s(a)ds\right)=\beta_t(a)-a;$$
\item for $a\in\mathcal{D}(\delta)$, $\beta_t(a)\in\mathcal{D}(\delta)$ and
$$\frac{d}{dt}\beta_t(a)=\delta(\beta_t(a))=\beta_t(\delta(a));$$
\item for $a\in\mathcal{D}(\delta)$
$$\beta_t(a)-\beta_s(a)=\int_s^t\beta_r(\delta(a))dr=\int_s^t\delta(\beta_r(a))dr.$$
\end{enumerate}
\end{lemma}
\vspace{0.001cm}

\begin{proof}[Proof of Theorem \ref{HY1}] By (1) and (2) of Lemma \ref{beta_t}, if $a\in\A$, we define $a_t:=\frac1t\int_0^t\beta_s(a)ds$, then $a_t\in\mathcal{D}(\delta)$ for $t \in {\mb R}$ and $a_t\to a$ as $t\to 0$. We conclude $\overline{\mathcal{D}(\delta)}=\A$.

In order to prove that $\delta$ is closed, let $\{a_n\}$ in $\mathcal{D}(\delta)$ such that $a_n\to a$ and $\delta(a_n)\to w$ as $n\to\infty$. By (4) of Lemma \ref{beta_t},
$$\beta_t(a_n)-a_n=\int_0^t\beta_s(\delta (a_n))ds.$$
Considering the limit on both sides of the equality and using again $(4)$ of Lemma \ref{beta_t}, we obtain
$$\beta_t(a)-a=\int_0^t\beta_s(w)ds.$$
Dividing by $t\neq 0$ and taking the limit as $t\to0$,  we conclude by $(1)$ of Lemma \ref{beta_t} that $a\in\mathcal{D}(\delta)$ and $\delta(a)=w$, i.e. $\delta$ is closed.

{If $\lambda=0$, the inequality is obvious. Now we consider $\lambda>0$ and define the operator}
$$R_{\lambda}(a):=\int_0^{\infty}\e^{-\lambda t}\beta_t(a)dt.$$
The continuity of $t\mapsto\beta_t(a)$ for every $a\in\A$ and the uniform boundedness of $\beta_t$ in $t$ guarantee that the above operator is well-defined and
$$\|R_{\lambda}(a)\|\leq\frac1{\lambda}\|a\|.$$

Moreover, {$(\lambda I-\delta)(R_{\lambda}(a))=a$, for every $a\in \A$ and $R_{\lambda}((\lambda I-\delta) (a))=a$}, for every $a\in \D(\delta)$. Indeed, the right hand side of the following
\begin{align*}\frac{\beta_h-I}{h}(R_{\lambda}(a))&=\frac1h\int_0^{\infty}\e^{-\lambda t}\left[\beta_{t+h}(a)-\beta_t(a)\right]dt\\
&=\frac{\e^{\lambda h}-I}{h}\int_0^{\infty}\e^{-\lambda t}\beta_t(a)dt-\frac{\e^{\lambda h}}h\int_0^h\e^{-\lambda t}\beta_t(a)dt
\end{align*}
tells us that $R_{\lambda}(a)\in\mathcal{D}(\delta)$ and it converges to $\lambda R_{\lambda}(a)-a$ for every $a\in\A$ and $\lambda>0$. Thus, $(\lambda I-\delta)R_{\lambda}=I$.

By the closedness of $\delta$ and again by Lemma \ref{beta_t},  we obtain also the other equality. Indeed,
\begin{align*}R_{\lambda}(\delta(a))&=\int_0^{\infty}\e^{-\lambda t}\beta_t(\delta(a))dt=\int_0^{\infty}\e^{-\lambda t}\delta(\beta_t(a))dt\\
&=\lim_{y\to\infty}\int_0^y\e^{-\lambda t}\delta(\beta_t(a))dt=\lim_{y\to\infty}\delta\left(\int_0^y\e^{-\lambda t}\beta_t(a)dt\right)\\
&=\delta\left(\int_0^{\infty}\e^{-\lambda t}\beta_t(a)dt\right)=\delta(R_{\lambda}(a)).
\end{align*}
Hence, $R_{\lambda}$ is the inverse of $\lambda I-\delta$ and the conditions on the spectrum are verified.

{The case when $\lambda <0$ can be handled in very similar way, by defining the operator $R_{\lambda}(a)$ as
 $$R_{\lambda}(a):=\int_0^{\infty}\e^{\lambda t}\beta_{-t}(a)dt.$$}
\end{proof}

In order to prove that a closed weak *-derivation is the infinitesimal generator of uniformly bounded, $\tau_n-$continuous group of weak *-automorphisms further assumptions on $\delta$ are needed.

\begin{thm}\label{HY2} Let $\delta:\mathcal{D}(\delta)\subset\A_{\bou}\to\A[\|\cdot\|]$ be a closed weak *-derivation on a *-semisimple Banach quasi *-algebra $(\A,\Ao)$. Suppose that $\delta$ verifies the same conditions on its spectrum of Theorem \ref{HY1} and $\Ao$ is a core for every multiplication operator $\hat{L_a}$ for $a\in\A$, i.e. $\hat{L}_a=\overline{L}_a$. Then $\delta$ is the infinitesimal generator of a uniformly bounded, {$\tau_n$-continuous} group of weak *-automorphisms of $(\A,\Ao)$.
\end{thm}
\begin{proof}
We want to show that the norm limit
$$\beta_t(a):=\lim_{n\to\infty}\left(I-\frac tn\delta\right)^{-1}(a)$$
gives us a uniformly bounded, {$\tau_n$-continuous} weak *-automorphisms group of $(\A,\Ao)$.

The existence of this limit can be derived by applying the theory of $C_0$-semigroups in Banach spaces \cite[Chap.12]{HP}.
Moreover, the map $t\in {\mb R}\to \beta_t(a)$ is norm continuous since \cite[p.362]{HP}  the convergence is uniform in every finite interval $[0,t_0]$.

By the condition on the spectrum of $\delta$, $\beta_t$ is, for every $t\in {\mb R}$, a bounded operator in $\A$ and all its powers are well defined. By the lower bound condition we obtain, for every $n\in\mathbb{N}^*$ ,
$$
\left\|\left(I-\frac tn\delta\right)^{-n}\!\!\!\!(a)\right\|\leq\|a\|,\quad\forall a\in\A$$
Hence passing to the limit we have $\|\beta_t(a)\|\leq\|a\|$ for {every $a\in\A$}.

Let $t\in\mathbb{R}$ be fixed. Then $\beta_t$ is a continuous linear and bijective operator. Moreover, $\beta_t$ preserves the involution, i.e. $\beta_t(a)^*=\beta_t(a^*)$ for every $a\in\A$. Further $\delta$ commutes with all its negative powers, so for every $a\in\A$ $\beta_t(\delta(a))=\delta(\beta_t(a))$.

By the fact that $\beta_t(a)$ is the solution of the Cauchy problem $\beta'_t(a)=\beta_t(\delta(a))$ with initial condition $\beta_0(a)=a$, we achieve the group property, i.e. $\beta_{t+s}(a)=\beta_t(\beta_s(a))$ for every $a\in\A$, $t,s\in\mathbb{R}$.

The set of analytic elements, i.e. the set of all elements $a\in\mathcal{D}(\delta^n)$, for every $n\in {\mb N}$, such that the power series
$$z\in\mathbb{C}\mapsto \sum_{n=0}^{\infty}\frac{z^n}{n!}\delta^n(a)\in\A$$
is well defined and analytic on a neighbourhood of the origin, is dense in $\A$ by \cite[Thm. 2.7]{Pazy}.

The last property we are going to prove is that $\beta_t$ is a weak *-automorphism, i.e. $\beta_t(a)\wmult \beta_t(b)$ is well defined if, and only if, $a\wmult b$ is well defined and, in this case, $\beta_t(a\wmult b)=\beta_t(a)\wmult \beta_t(b)$.

By the hypotheses, $\mathcal{D}(\delta)\subset\A_{\bou}$ is a partial *-algebra with respect to the weak multiplication $\mult$. By the boundedness of $a,b$, we can rewrite the weak Leibnitz rule as
$$\delta(a\wmult b)=\delta(a)\wmult b+a\wmult\delta(b).$$

Now suppose that $a,b$ are analytic elements. Therefore $a,b\in\mathcal{D}(\delta^k)$ for every $k\in\mathbb{N}$ and $\delta^k(a),\delta^k(b)\in\A_{\bou}$ for every $k$ because $\delta^k(a)\in\mathcal{D}(\delta^{k+1})\subset\mathcal{D}(\delta)\subset\A_{\bou}$. Hence, all the products $\delta^m(a)\wmult\delta^n(b)$ are well-defined for every $n,m\in\mathbb{N}$.

By the above argument, it is easy to prove by induction that
$$\delta^n(a\wmult b)=\sum_{k=0}^n\binom{n}{k}\delta^{n-k}(a)\wmult\delta^k(b)$$
for every $a,b$ analytic elements. In a very standard way we achieve the weak *-automorphism property in the case $a,b$ are analytic elements.

Using the density of the set of analytic elements and the boundedness of the elements one proves the equality
$$\beta_t(a\wmult b)=\beta_t(a)\wmult\beta_t(b),\quad a,b\in\A_{\bou}.$$

Let $a\in\A$ and $b\in\A_{\bou}$. Approximating an unbounded element $a$ through a sequence $a_n$ of bounded elements, the weak product $a\wmult b$ can be approximated by the sequence $a_n\wmult b$ and we get
$$\beta_t(a\wmult b)=\beta_t(a)\wmult\beta_t(b)\quad\text{for}\;a\in\A,b\in\A_{\bou}.$$

Suppose now that both $a,b\in\A$ are unbounded. By hypothesis, $\A_{\bou}$ is a core for every $L_a$, then there exists a sequence $\{b_n\}\in\A_{\bou}$ that norm converges to $b$ such that $\|a\wmult b_n-a\wmult b\|$ vanishes as $n$ increases. By norm continuity of $\beta_t$ we achieve the weak automorphism property for $\beta_t$, i.e.
$$\beta_t(a\wmult b)=\beta_t(a)\wmult\beta_t(b)\quad\forall a,b\in\A.$$
\end{proof}

\berem The additional hypotheses of Theorem \ref{HY2} are satisfied by the \textit{weak derivative in $L^p(I,d\lambda)$}, where $I=[0,1]$ and $\lambda$ is the Lebesgue measure.

In this case $\mathcal{D}(\delta)=W^{1,p}(I,d\lambda)$ and it is well known (see \cite[Thm 8.8]{Brezis}) that if $u\in W^{1,p}(I,d\lambda)$ then $u\in L^{\infty}(I,d\lambda)$ and there exists $c>0$ such that
$$\|u\|_{\infty}\leq c\|u\|_{1,p}.$$
{
Moreover, in the case $p=2$, it is possible to show that every weak *-automorphism of $L^2(I,d\lambda)$, coupled for instance with $L^{\infty}(I,d\lambda)$ or $\mathcal{C}(I)$, is automatically $\tau_n-$continuous with the same strategy employed in \cite{AT}. Indeed, if $\theta$ is a weak automorphism of $L^2(I,d\lambda)$, then $\theta$ is an intertwining operator with the couple $(R_x,R_{\theta(x)})$, i.e. $\theta\circ R_x=R_{\theta(x)}\circ\theta$ for every $x\in\Ao$. By Lemma \ref{im_theta}, $\theta(\A_{\bou})=\A_{\bou}$ thus the operator $\overline{R_{\theta(\psi)}}$ is everywhere defined and continuous.
}
\enrem

\section{Examples and applications}\label{Sec6}
In this section we present some examples of weak *-derivations and one-parameter groups generated by them. 

Consider again the example of inner qu*-derivations; i.e., $\delta:\Ao\to\A$ is a densely defined derivation determined as $\delta_h(x):=i[h,x]$ for a \textit{self-adjoint} element $h\in\A$, i.e. $h=h^*$ and $\sigma(h)\subset\mathbb{R}$.

\subsection{Inner qu*-derivations}

\textbf{Case 1}: Suppose first that $h$ is a \textit{bound\-ed element}. As we have already seen, in this case $\odelta_h(x)$ is \textit{continuous}.

Like in the classical case, what we would expect is a continuous one-parameter group $\{\beta_t\}_{t\in\mathbb{R}}$ of weak *-automorphisms of $(\A,\Ao)$ of the form
$$\beta_t(a)=\e^{ith}\wmult a\wmult \e^{-ith}\quad\text{for all}\,\;t\in\mathbb{R}.$$

Suppose that $(\A,\Ao)$ is a *-semisimple Banach quasi *-algebra with unit $\id$. Then we define by Taylor series $\e^{ith}$ as
$$\e^{ith}:=\sum_{n=0}^{\infty}\frac{(ith)^n}{n!},$$
where the series on the right hand side converges with respect to $\|\cdot\|_b$.
We stress the fact that $h^n$ is the weak product of $h$ with itself $n$ times. The above series is well defined, the exponential $\e^{ith}\in\A_{\bou}$ and all the known properties remain valid.

For each $t\in\mathbb{R}$, $\beta_t(a):=\e^{ith}\wmult a\wmult\e^{-ith}$ is a weak *-automorphism of $(\A,\Ao)$. We notice that by the separate continuity of multiplication and the *-semisimplicity of $(\A,\Ao)$ the use of brackets is not needed.
\vspace{0.05cm}

If we fix $t\in\mathbb{R}$, then it is routine to prove that $\beta_t$ is a linear map preserving the weak multiplication when defined.

Its inverse is given by $\beta_t^{-1}(a)=\e^{-ith}a\e^{ith}=\beta_{-t}(a)$ and $\beta:t\to\beta_t$ is in fact a weak *-automorphism group of $(\A,\Ao)$.
\vspace{0.05cm}

Self-adjoint elements can be characterized as those elements such that $\|\e^{ith}\|_{\bou}=1$. Following \cite[Proposition 2.4.12]{Dales}, we have that, for a *-semisimple quasi *-algebra $(\A,\Ao)$ and $h\in\A_{\bou}$ such that $h=h^*$, $\sigma(h)\subset\mathbb{R}$ if, and only if, $r_{\bou}(\e^{ith})=1$. Noticing that $(\e^{ith})^*=\e^{-ith}$ and, moreover, $\e^{ith}\wmult\e^{-ith}=\id=\e^{-ith}\wmult\e^{ith}$, we have that $\e^{ith}\in\A_{\bou}$ is normal and then $r(\e^{ith})=1=\|\e^{ith}\|_{\bou}$. Therefore we conclude that $\{\beta_t\}_{t\in\mathbb{R}}$ is uniformly bounded in $t$ by the following computations $\|\beta_t(a)\|\leq\|\e^{ith}\|_{\bou}^2\|a\|=\|a\|$.
\medskip

By standard computations it is easy to check that $\{\beta_t\}_{t\in\mathbb{R}}$ is really a norm continuous one-parameter group, i.e. $\beta_0(a)=a=\text{Id}(a)$, $\beta_{t+s}(a)=\beta_t\circ \beta_s(a)$ and $\|\beta_t(a)-a\|$ vanishes as $t\to0$, for every $a\in\A$.

We now compute the infinitesimal generator of $\{\beta_t\}_{t\in\mathbb{R}}$. What we expect is the closure of the inner qu*-derivation $\delta_h$ for $h\in\A_{\bou}$.
Indeed, it is straightforward to prove that $\frac{d}{dt}_{| t=0}\e^{ith}=ih$, so

$$\lim_{t\to0}\frac{\beta_t(a)-a}{t}=\lim_{t\to0}\frac{\e^{ith}\wmult a\wmult \e^{-ith}-a}{t}=ih\wmult a-ih\wmult a=\odelta_h(a)$$
for every $a\in\A=\mathcal{D}(\odelta_h)$. Therefore $\odelta_h$ is everywhere defined and continuous.

{
\berem Note that $\odelta_h$ is everywhere defined, that is $\mathcal{D}(\odelta_h)\not\subset\A_{\bou}$. Hence the hypothesis on the boundedness of $\mathcal{D}(\delta)$ is sufficient, but not necessary, to obtain a uniformly bounded norm continuous one-parameter group of weak *-automorphisms.
\enrem
}
\medskip

\textbf{Case 2}: We now consider the case in which $h$ is self-adjoint, as before, but \textit{unbounded}, i.e. $h\in\A\setminus\A_{\bou}$.

By definition given in Section 2, $\lambda\in\rho(h)$ if, and only if, $\lambda\in\rho(\overline{L}_h)\cap\rho(\overline{R}_h)$. We suppose that the element $h$ verifies the following condition
$$\|(h+i\gamma)^{-1}\|_{\bou}\leq\frac1{|\gamma|},\quad\gamma\in\mathbb{R}.$$
This, in turn, implies that
\begin{align*}
\|(\overline{L}_h+i\gamma I)^{-1}\|_{\mathcal{B}(\A)}&\leq\frac1{|\gamma|},\quad\gamma\in\mathbb{R}\\
\|(\overline{R}_h+i\gamma I)^{-1}\|_{\mathcal{B}(\A)}&\leq\frac1{|\gamma|},\quad\gamma\in\mathbb{R}.
\end{align*}
For every $t\in\mathbb{R}$, $it\in\rho(h)$ implies $it\in\rho(\overline{L}_h)\cap\rho(\overline{R}_h)$. Then there exists $\{U_L(t)\}_{t\in\mathbb{R}}$ strongly operator continuous one-parameter group such that $\|U_L(t)\|_{\mathcal{B}(\A)}\leq1$ for every $t\in\mathbb{R}$ and $\overline{L}_h$ is the infinitesimal generator of $\{U_L(t)\}_{t\in\mathbb{R}}$ (see \cite{Kato}).

Similarly, there exists a strongly continuous one-parameter group $\{U_R(t)\}_{t\in\mathbb{R}}$ such that $\|U_R(t)\|_{\mathcal{B}(\A)}\leq1$ for every $t\in\mathbb{R}$ and $\overline{R}_h$ is the infinitesimal generator of $\{U_R(t)\}_{t\in\mathbb{R}}$
\smallskip

Let us define
$$u_L(t):=U_L(t)(\id)\quad\text{and}\quad u_R(t):=U_R(t)(\id).$$
Since both $u_L(t)$ and $u_R(t)$ are solutions of the same differential equation $\frac{d u}{dt}=ihu$ with boundary condition $u(0)=\id$ in $\A$, $u_L(t)=u_R(t)$ for every $t\in\mathbb{R}$.

Hence we define
$$\e^{ith}:=u_L(t)=u_R(t),\quad t\in\mathbb{R}.$$
The exponential is a bounded element of $(\A,\Ao)$. Indeed, by \cite[Lemma 2.5.3]{FragCt}, it is easy to check that
$$\left(I-\frac{it}{n}\overline{L}_h\right)^{-n}(x)=\left(\left(I-\frac{it}{n}\overline{L}_h\right)^{-n}\id\right)x,\quad\forall x\in\Ao.$$
By the assumption on $h$, the element $(I-\frac{it}{n}\overline{L}_h)^{-n}\id$ is \textit{left-bounded} with the bound  not depending neither on $n$ nor on the element $e$. Therefore it is possible to extend the above equality for generic elements in $\A$
$$\left(I-\frac{it}{n}\overline{L}_h\right)^{-n}(a)=\left(\left(I-\frac{it}{n}\overline{L}_h\right)^{-n}\id\right)\wmult a,\quad\forall a\in\A.$$
Hence, by the strong continuity of $U_L(t)$, we achieve
$$U_L(t)a:=\lim_{n\to\infty}\left(I-\frac{it}{n}\overline{L}_h\right)^{-n}a=\lim_{n\to\infty}\left(\left(I-\frac{it}{n}\overline{L}_h\right)^{-n}\id\right)a$$
and $\|U_L(t)a\|\leq\|a\|$ for every $a\in\A$.

Analogously, $U_R(t)\id$ is \textit{right-bounded} and $\|U_R(t)a\|\leq\|a\|$. Then we conclude that $\e^{ith}$ is \textit{bounded}.

We have $U_L(t)a=U_L(t)\id\wmult a$ for every $a\in\A$. This is surely true for bounded elements because they are solutions of the same equation with the same boundary condition. Therefore, by separate continuity of the multiplication, the equality is true for every $a\in\A$. Hence we obtain $\|\e^{ith}\|_{\bou}\leq1$ and the group property $\e^{ith}\e^{ish}=\e^{i(t+s)h}$ for every $t,s\in\mathbb{R}$.

Since
$$\left(\left(I-\frac{it}n\overline{L}_h\right)^{-1}(a)\right)^*=\left(I+\frac{it}{n}\overline{R}_h\right)^{-1}(a^*),\quad\forall a\in\A,$$
$u_L(t)^*=u_r(-t)=u_L(-t)$. Then $\|\e^{ith}\|_{\bou}=1$.
\medskip

Defining $\beta_t(a):=\e^{ith}\wmult a\e^{-ith}$, we already know that $\{\beta_t\}_{t\in\mathbb{R}}$ is uniformly bounded continuous one-parameter group of continuous weak *-automorphisms. {The infinitesimal generator is given by the closure of the weak *-derivation $\delta_h(x)=i[h,x]$ for $x\in\Ao$. This closure is given by
$$\overline{\delta}_h(a)=\lim_{t\to0}\frac{\beta_t(a)-a}t=\frac{\e^{ith}\wmult a\wmult\e^{-ith}-a}t=i(h\wmult a-a\wmult h)$$
when $a$ is \textit{bounded}.}

\subsection{A physical example: quantum lattice systems}\label{sect_phys} 
The study of derivations and automorphism is important for physical applications to quantum systems with infinitely many degrees of freedom as, for instance, spin lattice systems. Without giving full details (for which we refer to \cite{Ant1,Bag1,Bag4,TW}) we give an outline of their mathematical description and show how the ideas developed in this paper may give some help when dealing with them.

Let $V$ is a finite region of a $d-$dimensional lattice and $\A_V$  the C*-algebra generated by the Pauli operators $\vec{\sigma}_p=(\sigma_p^1, \sigma_p^2, \sigma_p^3)$ at each point $p$ of the finite region $V$ (the number of points of $V$ is indicated by $|V|$) and by the identity matrix $I_p\in M_2(\mathbb{C})$. It is easy to show that $\A_V$ is isomorphic to $M_{2^{|V|}}(\mathcal{H}_V)$, where $\mathcal{H}_V=\otimes_{p\in V}{\mb C}_p^2$, and ${\mb C}_p^2$ is the 2-dimensional space at $p\in V$.

If $V\subset V'$, then there exists a natural embedding $\mathfrak{A}_V\hookrightarrow\mathfrak{A}_{V'}$, defined in obvious way. Hence $\Ao:=\cup_V\mathfrak{A}_V$ is a C*-algebra, called the C*-algebra of local observables; its norm is denoted by $\|\cdot\|_0$.

To any infinite sequence $\{n\}=\{n_i\}_{i=1}^\infty$ of unit vectors in $\mathbb{R}^3$ there corresponds a state $\ipp{}{\{n\}}$, constructed as in \cite[Section 11.3.1]{Ant1}. This state determines (GNS construction) a *-representation of $\Ao$ defined on the domain $\mathcal{D}^0_{\{n\}}=\ipp{\Ao}{\{n\}}$ whose completion is denoted by $\mathcal{H}_{\{n\}}$.
Then one can define a family of vectors
$$\left\{\ipp{}{\{m\},\{n\}}=\ipp{\otimes_p}{m_p,n_p};m_p=0,1,\sum_pm_p<\infty\right\}$$
which constitutes an orthonormal basis of ${\mc H}_{\{n\}}$. Each vector ($\ipp{}{\{m\},\{n\}}$ is obtained by flipping a finite number of spins in the {\em ground state} $\ipp{}{\{n\}}$.

Then, an unbounded self-adjoint operator $M$ acting on $\mathcal{H}_{\{n\}}$ is defined by
$$\ipp{M}{\{m\},\{n\}}=\ipp{\left(\sum_pm_p\right)}{\{m\},\{n\}}.$$

Roughly speaking, $M$ counts the number of flipped spins in $\ipp{}{\{m\},\{n\}}$ with respect to the ground state $\ipp{}{\{n\}}$.

Note that $M$ is strictly depending on the chosen sequence $\{n\}$. We set $\pi_{\{n\}}:\Ao\to\mathcal{L}^{\dagger}(\mathcal{D}_{\{n\}})$ to be the GNS *-representation defined by $\{n\}$ and we suppose that $\pi_{\{n\}}$ is faithful.
The operator $M$ is a number operator. Therefore, the operator $\e^M$ is a densely
defined self-adjoint operator. Let $\D$ denote its domain. Then $\D$ can be made into a Hilbert space,
denoted by $\H_M$, in canonical way. The norm in $\H_M$, in fact, is given by
$\| f\|_M =\|\e^M f \|, f\in \H_M .$

Let us assume that, for every $x\in \Ao$, both $\|\e^{M}\pi_{\{n\}}(x)\e^{-M}\|_0$ and $\|\e^{-M}\pi_{\{n\}}(x)\e^{M}\|_0$ are finite. Then the completion  $\A$ of $\Ao$ with respect to the norm
$$\|x\|:= \|\e^{-M}\pi_{\{n\}}(x)\e^{-M}\|$$ is a *-semisimple Banach quasi *-algebra.
Assuming that $h_V$ is the Hamiltonian of the finite volume system, then $h_V\in\Ao$ and then $\e^{ith_V}\in\Ao$. Now we define
$$\delta_V(x):=i[h_V,x],\quad\forall x\in\Ao.$$

By Proposition \ref{clos_inn}, $\delta_V$ is closable and, by $h_V\in\Ao$, it is actually continuous. Hence $\delta_V$ is infinitesimal generator of uniformly bounded norm continuous one parameter group of norm continuous weak *-automorphisms
$$\alpha^V_t(a)=\e^{ith_V}\wmult a\wmult \e^{-ith_V},\quad\forall a\in\A.$$
The interesting point comes when considering the so-called thermodynamical limit of the local dynamics; i.e. the $\lim_{|V|\to \infty} \delta_V$. This limit, in general, fails to exists in the C*-algebra topology of $\Ao$. It is here that the Banach quasi *-algebra structure plays a role, by taking the completion with respect to the norm $\|\cdot\|$ of $\A$. As shown in \cite{Ant1, Bag4}, under certain conditions, this limit exists and defines a weak *-derivation $\delta$ of $(\A,\Ao)$ which generates a one parameter group of *-automorphisms.

\section*{Concluding remarks}

Derivations are widely studied in mathematics for their importance in describing physical systems (see Section \ref{sect_phys}) and  for their own  interest. In this paper we investigated the problem of giving a proper definition of derivation in the framework of (mainly, *-semisimple) Banach quasi *-algebras

It would be interesting to examine the case of locally convex quasi *-algebra. In the aforementioned example the number operator $M$ is strictly depending on the state $\{n\}$, thus it is more convenient to consider a locally convex topology taking into account the entire family of states. This gives birth to a locally convex quasi *-algebra where a thermodynamical limit, possibly independent of the the representations, could lives.  We leave this question to future papers.
\bigskip

{\bf{Acknowledgement:} }
This work has been done in the framework of the project ''Alcuni aspetti di teoria spettrale di operatori e di algebre; frames in spazi di Hilbert rigged'', INDAM-GNAMPA 2018.


\begin{thebibliography}{99}

\bibitem{AT} {\sc M. S. Adamo, C. Trapani}, \emph{Representable and continuous functionals on a Banach quasi $^{\ast}-$algebra}, Mediterr. J. Math. (2017) 14: 157.

\bibitem{Ant1} {\sc J.-P. Antoine, A. Inoue, C. Trapani}, {\em Partial *-Algebras and their Operator Realizations}, Math. Appl., 553, Kluwer Academic, Dordrecht, 2003.

\bibitem{Ant3} {\sc J.-P. Antoine, A. Inoue, C. Trapani}, {\em Spatiality of *-derivations of partial O*-algebras}, J. Math. Phys. 36, 3743 (1995)

\bibitem{Ant4} {\sc J.-P. Antoine, A. Inoue, C. Trapani}, {\em Spatial theory of *-automorphisms on partial O*-algebras}, J. Math. Phys. 35, 3059 (1994)

\bibitem{Ant2} {\sc J.-P. Antoine, A. Inoue, C. Trapani}, {\em O*-dynamical systems and *‐derivations of unbounded operator algebras}, Math. Nachr., 204: 5-28 (1999).

\bibitem{Bag1} {\sc F. Bagarello, A. Inoue, C. Trapani}, {\em Derivations of quasi *-algebras}, International J.Math. Math. Sciences 21 (2004), 1077-1096.

\bibitem{Bag2} {\sc F. Bagarello, A. Inoue, C. Trapani}, {\em Representations and derivations of quasi *-algebras induced by local modifications of states}, J. Math. Anal. Appl., 356 (2009), 615-623.

\bibitem{Bag4} {\sc F. Bagarello, C. Trapani}, {\em The Heisenberg dynamics of spin systems: a quasi *-algebras approach}, J. Math. Phys. 37, 4219-4234 (1996)

\bibitem{Bag3} {\sc F. Bagarello, C. Trapani}, {\em $\mathrm{CQ}^{\ast}-$algebras: Structure properties}, Publ. RIMS, Kyoto Univ. {\bf 32} (1996), 85-116.

\bibitem{Brat1} {\sc O. Bratteli}, {\em Derivation, dissipation and group actions on C*-algebras}, Lecture Notes in Mathematics, vol.1229, Springer, Berlin (1986).

\bibitem{Brat2} {\sc O. Bratteli, D. W. Robinson}, {\em Unbounded derivations of C*-algebras}, Comm. Math. Phys. 42 (1975) n.3, 253--268.

\bibitem{BratRob} {\sc O. Bratteli, D. W. Robinson}, {\em Operator Algebras and Quantum Statistical Mechanics I}, Theoretical and Mathematical Physics, Springer-Verlag.

\bibitem{Brezis} {\sc H. Brezis}, {\em Functional Analysis, Sobolev Spaces and Partial Differential Equations}, Universitext, Springer-Verlag New York 2011.

\bibitem{Dales}{\sc H. Garth Dales}, {\em Banach algebras and automatic continuity}, Oxford Science Publications, Oxford, 2000.

\bibitem{Frag2} {\sc M. Fragoulopoulou, C. Trapani, S. Triolo}, {\em Locally convex quasi *-algebras with sufficiently many *-representations}, J. Math. Anal. Appl. {\bf 388} (2012) 1180 - 1193.

{
\bibitem{FragCt} {\sc M. Fragoulopoulou, C. Trapani}, {\em Locally convex quasi *-algebras and their representations}, in preparation.}


\bibitem{HP} {\sc E. Hille, R. Phillips}, {\em Functional analysis and semi-groups}, vol. 31, part 1, American Mathematical Society, 1996.

\bibitem{Kato} {\sc T. Kato}, {\em Perturbation theory for linear operators}, Classics in Mathematics 132, Springer- Verlag Berlin Heidelberg 1995.

\bibitem{Kish} {\sc A. Kishimoto}, {\em Dissipations and derivations}, Comm. Math. Phys. 47 (1976), n. 1, 25--32.


\bibitem{Pazy} {\sc A. Pazy}, {\em Semigroups of linear operators and applications to partial differential equations}, Applied Mathematical Sciences 44, Springer-Verlag New York 1983.

\bibitem{Ring} {\sc J. Ringrose}, {\em Automatic continuity of derivations of operator algebras}, J. London Math. Soc. {\bf 5} (1972), 4

\bibitem{Sakai} {\sc S. Sakai}, {\em Operator Algebras in Dynamical Systems}, Cambridge Univ. Press, Cambridge 1991.

\bibitem{Sin} {\sc A.M. Sinclair}, {\em Automatic Continuity of Linear Operators}, London Mathematical Society, Lecture Notes Series 21, Cambridge Press 1976.

\bibitem{TW} {\sc W.Thirring, A. Wehrl},{\em On the mathematical structure of the B.C.S.-model.}, Comm. Math. Phys. 4 (1967), no. 5, 303-314.

\bibitem{Trap1} {\sc C. Trapani}, {\em *-Representations, seminorms and structure properties of normed quasi *-algebras}, Studia Math. 186 (2008) 47-75.

\bibitem{ct1} {\sc C. Trapani}, {\em Bounded elements and spectrum in Banach quasi *-algebras}, Studia Mathematica {\bf 172} (2006), 249-273.

\bibitem{weigt} {\sc M. Weigt}, {\em Derivations of $\tau$-measurable Operators}, Operator Theory:
Adv. Appl., {\bf 195}, (2009) 273–286

\bibitem{WZ1} {\sc M. Weigt, I. Zarakas}, {\em Unbounded Derivations of GB*-algebras}, Operator Theory Adv. Appl. {\bf 247}, (2015) 69--82.

\bibitem{WZ2} {\sc M. Weigt, I. Zarakas}, {\em Derivations of Fr\'{e}chet nuclear GB*-algebras}, Bull. Austral. Math. Soc. 92 (2015), 290--301.

\end{thebibliography}
\end{document}